\documentclass[a4paper]{article}
\usepackage{amssymb}
\usepackage{amsmath}
\usepackage{amsthm}
\usepackage{hyperref}
\usepackage{enumerate}
\usepackage{mathtools}
\usepackage{url}
\usepackage[T1]{fontenc}

\newtheorem{Theorem} {Theorem} [section]

\newtheorem{Lemma} [Theorem] {Lemma}
\newtheorem{Corollary} [Theorem] {Corollary}

\newtheorem{Problem} [Theorem] {Problem}

\newtheorem{Conjecture} [Theorem] {Conjecture}

\newcommand{\Ef}{{\mathbb E}}
\newcommand{\Ff}{{\mathbb F}}

\newcommand{\Pf}{{\mathbb P}}

\newcommand{\cC}{{\mathcal C}}

\newcommand{\cL}{{\mathcal L}}

\newcommand{\cS}{{\mathcal S}}

\newcommand{\PG}{\mathrm{PG}}
\newcommand{\eps}{\varepsilon}
\newcommand{\<}{\langle}
\renewcommand{\>}{\rangle} 
\renewcommand{\phi}{\varphi} 
\newcommand{\gauss}[2]{\genfrac{[}{]}{0pt}{}{#1}{#2}}
\newcommand{\coloneq}{\vcentcolon=}      

\title{Large $(k; r, s; n, q)$-sets in Projective Spaces}
\author{
Ferdinand Ihringer, Jacques Verstra\"{e}te
}
\date{11 Nov 2022}

\begin{document}
\maketitle

\begin{abstract}
  A $(k; r, s; n, q)$-set (short: $(r,s)$-set) of $\PG(n, q)$ is a set of points $X$
  with $|X| = k$ such that no $s$-space contains more than $r$ points of $X$.
  We investigate the asymptotic size of $(r, s)$-sets 
  for $n$ fixed and $q \rightarrow \infty$.
  In particular, we show the existence of 
  $(3, 2)$-sets of size $(1+o(1)) q^{3/2}$ for $n=6$,
  $(4, 2)$-sets of size $(1+o(1)) q^{\frac{n-1}{2}}$,
  and $(9, 2)$-sets of size $(1+o(1)) q^2$ for $n=4$.
  We also generalize a bound by Rao from 1947
  and show that an $(r,s)$-set has size at most $O(q^{\frac{n-e+1}{e}})$ if there 
  exist integers $d,e \geq 2$ such that $s=d(e-1)$ and $r=de-1$.
\end{abstract}

\section{Introduction}

Let $\PG(n, q)$ denote the finite projective space
whose {\it points} and {\it lines} are the $1$- and $2$-dimensional 
subspaces of $\Ff_q^{n+1}$, respectively.
A {\it cap} is a set $X$ of points in $\PG(n, q)$ such that no line 
intersects $X$ in more than $2$ points.
More generally, a set $X$ in $\PG(n, q)$ is commonly called {\it $(k; r, s; n, q)$-set}
if each $s$-space contains at most $r$ points of $X$ and $|X| = k$.
The present work is concerned with $(k; r, s; n, q)$-sets for fixed $n$ 
and asymptotic behavior for $q \rightarrow \infty$.
In the following we abbreviate $(k; r, s; n, q)$-set 
as {\it $(r, s)$-set} (as for us $n$ is fixed, $k$ not exact, 
and we consider asymptotics in $q$).
A $(2, 1)$-set is known as a {\it cap},
a $(r, r-1)$-set is known as a {\it generalized cap},
a $(n, n-1)$-set is known as an {\it arc} and 
corresponds to an {\it MDS code}.
Hirschfeld surveys upper and lower bounds of $(r,s)$-sets 
of maximal size in \cite{Hirschfeld1983}.
In other literature $(r, s)$-sets are known 
as {\it $(s, r)$-subspace evasive sets}, see also \cite{Guruswami2011,PR2004}.

An $(s,1)$-set in $\PG(n, q)$ has size at most $O(q^{n-1})$.
This bound is known to be tight for $n=2,3$.
Segre showed in 1959 \cite{Segre1959} (c.f.\ \cite{EB1999}) that the largest cap
has at least size $\Omega(q^{\lceil \frac23 n - \frac23 \rceil })$.
The first open case for caps is $\PG(4, q)$ for which the largest 
known constructions have size $(3+o(1))q^2$, see \cite{EB2004}.
Dvir and Lovett showed in \cite{DL2011} (Theorem 2.4 together with Claim 3.5)
that for $q$ sufficiently large,
there exists an $(n^s, s)$-set of size $\frac13 q^{n-s}$.
Also see \cite{ST2022} for a construction using random polynomials
(based on the method developed by Bukh in \cite{Bukh2015}).
In Section \ref{sec:lns} we give an explicit
$(n,1)$-set of size $(q-1)^{n-1}$ in $\PG(n, q)$.

In Section \ref{sec:bounds} we establish simple upper
and lower bounds on the size of largest $(r, s)$-sets
for $r \leq 2s-1$.
Sudakov and Tomon show in \cite[Theorem 1.2]{ST2022} that an $(r, s)$-set with 
$r \leq \frac{3}{2} s - 1$ has at most size $4s \cdot q^{\frac{n}{D(r, s)}}$ where $D(r, s) = \lfloor \frac{s}{2(r-s+1)}\rfloor$.
The following lower bound was observed for $s=r-1$ by Rao \cite[p. 136]{Rao1947}
and Bose \cite[Eq. 5.64]{Bose1947}. Their result has been 
rediscovered by Tait and Won in 2021 \cite[Th. 3.1]{TW2021}.
Otherwise, it appears to be new and improves on the bound by Sudakov and Tomon
(at least for $q \geq 4e$).

\begin{Theorem}\label{lem:upper_bnd1}
  Let $X$ be an $(r, s)$-set in $\PG(n, q)$ such that there exist
  integers $d, e \geq 2$ with $s = d(e-1)$ and $r=de-1$.
  Then $|X| \leq C_{d,e} q^{\frac{n-e+1}{e}}+e$ where $C_{d, e} = \sqrt[e]{e! (d-1) \frac{1+2q^{-1}}{1-2e q^{-1}}}$ for $q > 2e$.
  Particularly, $|X| \leq C_{2, e} q^{ \frac{n-e+1}{e}}+e = C_{2, e} q^{\frac{2n-s}{s+2}}+e$ for $s=r-1$ even, and
  $|X| \leq C_{d, 2} q^{\frac12 (n-1)}+e$ for $r=2s-1$. 
\end{Theorem}

In Section \ref{sec:planes}, we investigate the special case of planes.
In particular, we provide a construction for a $(3,2)$-set of size $\Theta(q^{3/2})$ in $\PG(6, q)$
which combines algebraic and probabilistic methods.
This improves on the trivial lower bound of $\Omega(q^{4/3})$.
We also show for $n = 4$ that $(9, 2)$-sets of size $(1+o(1)) q^{2}$
exist for $q$ sufficiently large, improving the result by Dvir and Lovett
for that case.
We conclude in Section \ref{sec:field_red} with product type constructions
for $(r,s)$-sets using field reduction.

\section{Lines} \label{sec:lns}

Throughout the whole text, we identify a vector of the form
$(1, x_1, \ldots, x_n)$ with the 
(affine) point $\< (1, x_1, \ldots, x_n)\>$ of $\PG(n, q)$.
We assume that $n \geq 2$.
Our main result on lines is as follows:

\begin{Theorem}\label{thm:lns}
  There exists a $(n,1)$-set of size $(q-1)^{n-1}$ in $\PG(n, q)$.
\end{Theorem}

We shall prove this in two lemmas.
Let $x \in \Ff_q^{n-1}$.
Define
\begin{align*}
  &F_0(x) = \prod_{i=1}^{n-1} x_i,
  &&F_i(x) = x_i F_0(x) ~\text{ for }~ i \in \{ 1, \ldots, n-1 \}.
\end{align*} 
Put $F(x) = (1, F_0(x), \ldots, F_{n-1}(x))$.
We denote $\Ff_q \setminus \{ 0 \}$ by $\Ff_q^*$
and write $\Ff_q^{*n-1}$ for $(\Ff_q^*)^{n-1}$.
Put 
$
  X = \{ F(x): x \in \Ff_q^{*n-1} \}.
$

\begin{Lemma}\label{lem:nr_pts}
  The curve $X$ has $(q-1)^{n-1}$ points.
\end{Lemma}
\begin{proof}
  Let $u \in X$. Then there exists an $x \in \Ff_q^{*n-1}$
  such that $u = F(x)$. As $x_1\cdots x_{n-1} \neq 0$, we have 
  $x = ( u_2/u_1, \ldots, u_n/u_1)$.
\end{proof}

\begin{Lemma}\label{lem:nr_triples}
  No line intersects $X$ in more than $n$ points.
\end{Lemma}
\begin{proof}
\newcommand{\Fx}{F_0(x)}
\newcommand{\Fy}{F_0(y)}
\newcommand{\Fz}{F_0(z)}
  Suppose that three distinct points $u,v,w \in X$ with $u=F(x)$, 
  $v=F(y)$, and $w=F(w)$ for $x,y,z \in \Ff_q^{*n-1}$ are collinear. %
  Then the following matrix has rank 2 (if the rank is 1, then $x=y=z$; 
  if the rank is 3, then $\< u, v, w\>$ is a plane):
  \begin{align*}
    \begin{pmatrix}
     1 & 1 & 1\\
     \Fx & \Fy & \Fz \\
     x_1 \Fx & y_1 \Fy & z_1 \Fz \\
     x_2 \Fx & y_2 \Fy & z_2 \Fz \\
     \vdots & \vdots & \vdots  \\
     x_{n-1} \Fx & y_{n-1} \Fy & z_{n-1} \Fz
    \end{pmatrix}.
  \end{align*}
  By subtracting $z_i$-times the second row from the $(i+1)$-th 
  row for $i \in \{ 1, \ldots, n-1\}$ and subtracting $\Fz$-times the first 
  row from the second, we obtain 
  \begin{align*}
    &M \coloneq \begin{pmatrix}
     1 & 1 & 1 \\
     \Fx-\Fz & \Fy-\Fz & 0 \\
     (x_1-z_1) \Fx & (y_1-z_1) \Fy & 0 \\
     (x_2-z_2) \Fx & (y_2-z_2) \Fy & 0 \\
     \vdots & \vdots  & \vdots \\
     (x_{n-1}-z_{n-1}) \Fx & (y_{n-1}-z_{n-1}) \Fy & 0 
    \end{pmatrix}.
  \intertext{ Consider}
    &M' \coloneq \begin{pmatrix}
     \Fx-\Fz & \Fy-\Fz \\
     (x_1-z_1) \Fx & (y_1-z_1) \Fy \\
     (x_2-z_2) \Fx & (y_2-z_2) \Fy \\
     \vdots & \vdots \\
     (x_{n-1}-z_{n-1}) \Fx & (y_{n-1}-z_{n-1}) \Fy
    \end{pmatrix}.
  \end{align*}
  The matrix $M$ has rank $2$ if and only if $M'$ has rank $1$. 
  This is precisely the case when all $(2 \times 2)$-submatrices
  $\tilde{M}$ of $M'$ have rank at most $1$, that is $\det(\tilde{M})=0$.
  Hence, we obtain that 
  \begin{align}
    & (\Fx - \Fz) \Fy (y_i - z_i) = (\Fy - \Fz) \Fx (x_i - z_i)
    \label{eq:dets2}
  \intertext{ for $i \in \{ 1, \ldots, n-1 \}$, and} 
    &(x_1-z_1)(y_i-z_i) = (y_1-z_1)(x_i-z_i)
    \label{eq:dets3}
  \end{align}
  for $i \in \{ 2, \ldots, n-1 \}$. Here we use that $\Fx, \Fy \neq 0$.
  
  \medskip 
  
  We want to show that no line contains more than $n$ points.
  
  If $\Fx = \Fy = \Fz$, then the equations in \eqref{eq:dets3}
  plus $\Fx = \Fy = \Fz$ determine the entries of $z$ as the solution 
  to polynomials of degree $n-1$. Hence, there
  are at most $n-3$ nontrivial solutions for $z$.
  
  Otherwise, Equation \eqref{eq:dets2} for $i=1$ 
  together with the equations in \eqref{eq:dets3}
  determines the entries of $z$ as the solution to polynomials
  of degree $n$. Hence, we have at most $n-2$ nontrivial solutions.
\end{proof}

Note that the points of $X$ lie on the affine curve 
\[ 
 X_1^n - X_2 \cdots X_n = 0,
\]
and (if we switch the first two coordinates) the projective curve
\[
  X_0^n - X_1 X_2 \cdots X_n = 0.
\]
It seems to be highly natural do study these curves over finite fields.
Note that the canonical affine part of the last curve, that is
\[
    \{ ( 1, \tfrac{1}{x_1 \cdots x_{n-1}}, x_1, x_2, \ldots, x_{n-1}): x \in \Ff_q^{*n-1} \},
\]
satisfies Lemma \ref{lem:nr_pts} and Lemma \ref{lem:nr_triples} as well.

\section{Simple Bounds on \texorpdfstring{$(r,s)$}{(r,s)}-Sets} \label{sec:bounds}

\begin{table}
\begin{center}
\begin{tabular}{l|lllllll}
$r - s$ & 1 & 2 & 3 & 4 & 5 & 6\\ \hline 
$s=1$ & $n{-}1$ &  &  &  &  \\
$s=2$ & $\frac{n-1}{2}$ & $n{-}2$ &  &  &  \\
$s=3$ & $\frac{n-2}{2}$ & $\frac{n-1}{2}$ & $n{-}3$ &  &  \\
$s=4$ & $\frac{n-2}{3}$ & $\frac{n-2}{2}$ & $\frac{n-1}{2}$ & $n{-}4$\\
$s=5$ & $\frac{n-3}{3}$ & $\frac{n-3}{2}$ & $\frac{n-2}{2}$ & $\frac{n-1}{2}$ & $n{-}5$\\
$s=6$ & $\frac{n-3}{4}$ & $\frac{n-2}{3}$ & $\frac{n-3}{2}$ & $\frac{n-2}{2}$ & $\frac{n-1}{2}$ & $n{-}6$\\
\end{tabular}
\end{center}
\caption{The upper bounds $O(q^m)$ on the size of a $(r, s)$-set from 
Theorem \ref{lem:upper_bnd1}, Lemma \ref{lem:upper_bnd4}, and Lemma \ref{lem:upper_bnd5}
for small $s$ and sufficiently large $n$. 
Here $m$ is the entry in the table, indexed by $r$ and $r-s$.}
\end{table}

We write $\gauss{n+1}{k+1}$ for the number
of $k$-spaces in $\PG(n, q)$. For our estimates, 
we will use that $\gauss{n+1}{k+1} = (1+o(1)) q^{(k+1)(n-k)}$
(as $q \rightarrow \infty$).
The following result was observed for $s=r-1$ 
by Gilbert \cite[I.4]{Gilbert1952} and Varshamov \cite{Varshamov1957}
and is well-known in coding theory as the Gilbert-Varshamov bound.

\begin{Lemma}\label{lem:lower_bnd2}
 Let $s+1 \geq r \geq 3$.
 There exists a $(r,s)$-set in $\PG(n, q)$
 of size $(\frac{r!-1}{r!} + o(1)) q^{n - s - s(n-s)/r}$ (as $q \rightarrow \infty$).
\end{Lemma}
\begin{proof}
  Delete points from $\PG(n, q)$ with probability $p$.
  Let $X$ denote the resulting set of points.
  For a random $s$-space $S$, we have
  \[
   \Pf(|X \cap S| \geq r+1) \leq \binom{\gauss{s+1}{1}}{r+1} p^{r+1}.
  \]
  Hence, if we take $p = (1+o(1))q^{-s(n-s)/r-s}$
  and let $\cS$ denote the set of $s$-spaces which meet $X$ in at least $r+1$ points, then
  \begin{align*}
    \Ef(|\cS|)
    \leq \gauss{n+1}{s+1} \binom{\gauss{s+1}{1}}{r+1} p^{r+1}
    = \frac{1}{(r+1)!} (1+o(1))q^{n-2s -s(n-s)/r }.
  \end{align*}
  Furthermore, we have 
  \begin{align*}
   \Ef(|X|) = p \gauss{n+1}{1} = (1+o(1)) q^{n-s - s(n-s)/r}.
  \end{align*}
  Hence, we can delete the at most $(r+1) \cdot \frac{1}{(r+1)!} (1+o(1))q^{n-s -s(n-s)/r }$
  points of $X$ in any $s$-space that intersects $X$ in 
  at least $r+1$ points and we obtain an $(r, s)$-set of size
  \[
   \frac{r!{-}1}{r!} (1+o(1))q^{ n-s-s(n-s)/r }. \qedhere
  \]
\end{proof}

The following statement was already observed by Gulati \cite{Gulati1971a}
for $s=r-1$.

\begin{Lemma}\label{lem:upper_bnd5}
 Let $m \geq 1$.
 Let $B$ be the maximal size of an $(r, s)$-set in $\PG(n, q)$.
 Then $|X|-m \leq B$ for an $(r+m, s+m)$-set $X$ in $\PG(n+m, q)$.
\end{Lemma}
\begin{proof}
  Let $S$ be the subspace spanned by $m$ points of $X$.
  The projection of $X \setminus S$ onto a complement of $S$
  is an $(r, s)$-set in $\PG(n, q)$.
\end{proof}

Recall for the following that an $(r,s)$-set 
with $r \leq s$ contains at most $r$ points.
If an $(r,s)$-set $X$ meets an $m$-space $S$ in $d$ points
of $X$ for some $d \leq r$, then in the quotient of $S$
we find an $(r-d,s-m-1)$-set $X'$ of size $|X|-d$. 
If $s-m-1 \geq r-d$, then $|X'| \leq r-d$, 
so $|X| \leq r$. We call an $(r,s)$-set $X$ {\it proper}
if no $m$-space contains $m+1+r-s$ (or more) points.
Any $(r,s)$-set $X$ with $|X| > r$ is proper.


\begin{proof}[Proof of Theorem \ref{lem:upper_bnd1}]
%
  Any $e=s/d+1$ points of $X$ span a subspace of dimension at
  most $e-1$.
  Let $Y_1, \ldots, Y_d \subseteq X$ sets of size $e$.
  Let $\tilde{Y}_i$ denote the set of points of $\PG(n, q)$
  in $\< Y_i \> \setminus X$, but not in the span of any $(e-1)$-subset 
  of $Y_i$.
  
  Put $y = |\bigcup_{i=1}^d Y_i|$.
  We claim that $\bigcap_{i=1}^d \tilde{Y}_i$ is empty:
  Otherwise, $\bigcup_{i=1}^d Y_i$ spans a subspace of dimension
  at most $(y - d + 1) - 1 = y-d$ that contains at least $y$ points
  of $X$. But $y-d+1+r-s = y$, so $X$ is not proper, thus $|X| \leq r$ 
  and we are done.
  
  Hence, a point $P$ not in $X$ lies in at most 
  $d-1$ sets $\tilde{Y}_i$.
  Double count $e$-tuples $Y \subseteq X$
  and points $P$ not in $X$ with $P \in \< Y\>$. We obtain that
  \[
    \binom{|X|}{e} \left(\gauss{e}{1} - \binom{e}{e-1} \gauss{e-1}{1}\right) 
    \leq (d-1) \left(\gauss{n+1}{1} - |X|\right).
  \]
  Using  $1 \leq \gauss{m}{1}/q^{m-1} \leq 1+2q^{-1} \leq 2$, we obtain 
  \[
   \frac{(|X|-e)^e}{e!} q^{e-1} \left(1 - 2eq^{-1}  \right) 
    \leq (1+2q^{-1}) (d-1) q^n.
  \]
  Rearranging for $|X|$ yields the assertion.
  The special cases are for $d=2$ and $e=2$.
\end{proof}

\begin{Lemma}\label{lem:upper_bnd4}
  Let $X$ be a $(r, s)$-set in $\PG(n, q)$ with $s \geq 1$.
  Then $|X| \leq r \left( \gauss{n-s+1}{1} + 1\right)$.
\end{Lemma}
\begin{proof}
  Fix a $(r-1)$-space $S$ which intersects $X$ in at least $s$ points.
  Then each $s$-space $T$ in $\PG(n, q)$ through $S$ contains at most
  $r$ elements of $X$. There are $\gauss{n-s+1}{1}$ such $T$.
\end{proof}

\section{Planes} \label{sec:planes}

\subsection{Plane Sets in \texorpdfstring{$\PG(6, q)$}{PG(6, q)}} \label{subsec:planes32}

We construct a $(3,2)$-set of size $(1+o(1)) q^{3/2}$ in $\PG(6, q)$.
The (trivial) lower bound from Lemma \ref{lem:lower_bnd2} is $\Omega( q^{4/3})$,
while the (trivial) upper bound from Theorem \ref{lem:upper_bnd1}
is $O( q^{ 5/2})$.
Put \[X = \{ (1, x, x^2, x^3, y, y^2, y^3): x, y \in \Ff_q \}.\]
Clearly, $|X| = q^2$.

\begin{Lemma}\label{lem:char_C4}
  Let $s_1, s_2, s_3, s_4 \in X$ be four pairwise distinct points in a plane
  of $\PG(6, q)$ with $s_i = (1, x_i, x_i^2, x_i^3, y_i, y_i^2, y_i^3)$.
  Then it holds that $|\{ x_1, x_2, x_3, x_4\}| = |\{ y_1, y_2, y_3, y_4\}| = 2$.
\end{Lemma}
\begin{proof}
  As $s_1, s_2, s_3, s_4$ lie in a plane,
  the matrix
  \[
   M = \begin{pmatrix}
      1 & 1 & 1 & 1 \\
      x_1 & x_2 & x_3 & x_4 \\
      x_1^2 & x_2^2 & x_3^2 & x_4^2 \\
      x_1^3 & x_2^3 & x_3^3 & x_4^3 \\
      y_1 & y_2 & y_3 & y_4 \\
      y_1^2 & y_2^2 & y_3^2 & y_4^2 \\
      y_1^3 & y_2^3 & y_3^3 & y_4^3 \\
   \end{pmatrix}
  \]
  has at most rank $3$.
  The top-left $4 \times 4$ submatrix of $M$
  is a Vandermonde matrix of rank at most $3$.
  Hence, if the $x_1, x_2, x_3$ are pairwise distinct,
  then $x_4 \in \{ x_1, x_2, x_3 \}$.
  If $x_4 = x_i$ for some $i \in \{1, 2, 3\}$,
  then also $y_4 = y_i$.
  But then $s_4 = s_i$ which is a contradiction.
  Hence, $|\{ x_1, x_2, x_3, x_4\}| \leq 2$,
  and similarly, $|\{ y_1, y_2, y_3, y_4\}| \leq 2$.
  If any of the inequalities above is not obtained,
  then again two $s_i$ are identical, so this does not happen.
\end{proof}

Let $G$ be a bipartite $C_4$-free graph with $(1+o(1)) q^{3/2}$
edges and $q$ vertices in each half of the bipartition.
Such graphs exists, for instance see \cite[\S3.1]{FS2013}.
Identify the points and lines of $G$ 
with two distinct copies of $\Ff_q$.
Put \[ X' = \{ (1, x, x^2, x^3, y, y^2, y^3): x, y \in \Ff_q,\, 
xy \text{ is an edge of } G \}.\]

\begin{Theorem}\label{thm:q32constr}
  We have $|X'| = (1+o(1))q^{3/2}$ and $X'$ is a $(3,2)$-set.
\end{Theorem}
\begin{proof}
  The graph $G$ has $(1+o(1))q^{3/2}$ edges,
  so $|X'| = (1+o(1))q^{3/2}$.
  Lemma \ref{lem:char_C4} says that four point $s_1, s_2, s_3, s_4$ in 
  $X$ which lie in a plane are precisely those
  with WLOG $x_1 = x_2$, $x_3 = x_4$, $y_4 = y_1$, and $y_2 = y_3$ 
  (in the notation of Lemma \ref{lem:char_C4}).
  This does not happen in $X'$ as 
  $(x_1, y_1), (x_1, y_2), (x_3, y_2), (x_3, y_1)$
  would correspond to a quadrangle in $G$.
\end{proof}
  
Now assume that $q = q_0^2$ for some prime power $q_0$.
In this case there exists a natural choice for $G$.
Let $G'$ be the incidence graph of $\PG(2, q_0)$.
It is well-known that the graph $G$ is $C_4$-free.
Remove all points on a fixed line $L$ from $G$,
and remove all lines on a fixed point on $L$ from $G$.
Then $G'$ has $q$ lines and $q$ points.
(If one wishes for a completely explicit construction,
then identifying $\Ff_q$ with $\Ff_{q_0}^2$ is
a natural choice.)
Let $X''$ be the point set $X'$ with $G'$ for $G$.

\begin{Corollary}
  For $q$ a square, we have $|X''| = q^{3/2}$ and $X''$ is a $(3,2)$-set.
\end{Corollary}

%

\subsection{Random \texorpdfstring{$(4,2)$}{(4,2)}-sets from Quadrics}

Theorem \ref{lem:upper_bnd1} gives $(\sqrt{2}+o(1)) q^{\frac{n-1}{2}}$
as an upper bound for $(3,2)$-sets in $\PG(n, q)$, that is a set of points with at most 3
in each plane. Here we show that one can obtain examples of that
size if we allow up to 4 points on a plane, so $(4,2)$-sets.
Note that for $(4,2)$-sets we only have Lemma \ref{lem:upper_bnd4}
which states an upper bound of $O(q^{n-2})$.


\begin{Theorem}
  Let $m \geq 2$ and put $n=2m-1$.
  There exists a $(4,2)$-set in $\PG(n, q)$ 
  of size $(1+o(1)) q^{\frac{n-1}{2}}$.
\end{Theorem}
\begin{proof}
  Let $Q_1, \ldots, Q_{m}$ be random irreducible quadrics.
  Put $X = \bigcap_{i=1}^{m} Q_i$.
  As $Q_i$ has $(1+o(1)) q^{n-1}$ points, 
  we find $\Ef(|X|) = (1+o(1))q^{n-m} = (1+o(1))q^{m-1}$.
  Recall that a conic is determined by 5 points,
  two conics intersect in at most $4$ points,
  and that there are $(1+o(1))q^5$ conics in $\PG(2, q)$.
  
  Let $C_1$ and $C_2$
  be random conics. Then
  \begin{align*}
   \Pf( |C_1 \cap C_2| > 4 ) = \Pf( |C_1 \cap C_2| = q+1) = (1+o(1)) q^{-5}.
  \end{align*}
  Let $\cC$ be the set of planes which intersect $X$
  in a conic plane. There are $(1+o(1)) q^{3(n-2)}$
  conic planes in $Q_1$. 
  Hence, 
  \[
    \Ef(|\cC|) = (1+o(1)) q^{3(n-2)} \cdot (1+o(1)) q^{-5(m-1)} = (1+o(1)) q^{m-4}.
  \]
  Hence, the expected number of points of $X$ in a
  conic plane is at most $(1+o(1))q^{m-3}$
  as each of them has at most $q+1$ points in $X$.
  
  Let $L_1$ and $L_2$ be random singular lines in two of the quadrics.
  Then $\Pf(L_1 = L_2) = (1+o(1)) q^{-3}$.
  Let $\cL$ be the set of lines which are completely contained in $X$.
  There are $\gauss{n+1}{2} = (1+o(1)) q^{2(n-1)}$ lines in $\PG(n-1, q)$,
  so 
  \[
    \Ef(|\cL|) = (1+o(1))q^{2(n-1)} \cdot (1+o(1))q^{-3m} = q^{m-4}.
  \]
  Hence, we can delete all points of $X$ in conic planes and 
  complete lines to obtain some $X'$ with $\Ef(|X'|) = (1+o(1)) \Ef(|X|) = (1+o(1))q^{m-1}$.
\end{proof}

\subsection{Random \texorpdfstring{$(9,2)$}{(9,2)}-sets from Cubic Curves}

Lemma \ref{lem:upper_bnd4} gives $O(q^2)$ as an upper bound for a $(9,2)$-set
in $\PG(4, q)$. Dvir and Lovett show that there exists a $(16,2)$-set. Here we show that cubic polynomials give rise to a $(9,2)$-set of size $(1+o(1)) q^2$.

\begin{Lemma}\label{lem:cub_red_cnt}
  In $\PG(2, q)$ the number of irreducible cubic curves is $(1+o(1)) q^9$,
  the number of reducible cubic curves containing a conic is $O(q^7)$,
  and the number of reducible cubic curves containing no conic is $O(q^6)$.
\end{Lemma}
\begin{proof}
  Nine points determine
    a cubic curve. As $9$ points in general position determine a cubic, the number is
    at least
    \[
        (1+o(1)) \frac{q^{2 \cdot 9}}{q^9} = (1+o(1))q^9.
    \]
    Next we estimate the number of reducible cubic curves.
    These consist of three lines, or a line and a conic.
    Recall that there are $O(q^2)$ lines and 
    $O(q^5)$ quadratic curves in $\PG(2, q)$.
\end{proof}

\begin{Lemma}\label{lem:cubic_int}
 Let $C_1, C_2$ two random irreducible cubic curves in $\PG(2, q)$.
 Then 
 \[
    \Pf(|C_1 \cap C_2| > 9) \leq (1+o(1)) q^{-9}.
 \]
\end{Lemma}
\begin{proof}
    By B\'ezout's Theorem, $|C_1 \cap C_2| \leq 9$
    or $C_1 = C_2$. Lemma \ref{lem:cub_red_cnt} shows the assertion.
\end{proof}

\begin{Theorem}
  There exists a $(9,2)$-set in $\PG(4, q)$ 
  of size $(1+o(1)) q^2$.
\end{Theorem}
\begin{proof}
  Let $C_1, C_2$ be random cubic surfaces.
  Put $X = C_1 \cap C_2$.
  As $|C_1| = (1+o(1)) q^{n-1} = |C_2|$, $\Ef(|X|) = (1+o(1)) q^2$.
  Let $Y_{i,j}$ be the set of planes which 
  intersects $C_1 \cap C_2$ in more than $9$ points
  and have an $i$-space (respectively, $j$-space) 
  as the largest subspace contained in $C_1$ (respectively, $C_2$).
  By Lemma \ref{lem:cub_red_cnt}, 
  $\Ef(|Y_{2,j}|) = \gauss{5}{3} q^{-9} < (1+o(1)) q^{-1}$ 
  for $j \in \{ 0, 1, 2\}$.
  By Lemma \ref{lem:cub_red_cnt},
  $\Ef(|Y_{1,j}|) = \gauss{5}{3} q^{-2-4} < 1+o(1)$
  for $j \in \{ 0, 1 \}$.
  By Lemma \ref{lem:cubic_int}, $\Ef(|Y_{0,0}|) \leq \gauss{5}{3} (1+o(1)) q^{-9} < 1 + o(1)$.
  Hence, we can delete the at most $(1+o(1)) q$ points
  of $X$ in $\bigcup Y_{i,j}$ and obtain a $(9,2)$-set
  of size $(1+o(1)) q^2$.
\end{proof}

\section{Product Constructions via Field Reduction} \label{sec:field_red}

We present a generic product construction for $(r,s)$-sets.

\begin{Lemma}
  Let $X$ be a proper $(r, r-1)$-set in $\PG(N{-}1, q)$
  and let $Y$ be an $(r, s)$-set in $\PG(M{-}1, q^{N})$.
  Then there exists an $(r, s)$-set in $\PG(NM{-}1, q)$
  of size $|X| \cdot |Y|$.
\end{Lemma}
\begin{proof}
  By field reduction, the elements of $Y$ correspond
  to $(N{-}1)$-spaces in $\PG(NM{-}1, q)$.
  Arbitrarily pick a copy of $X$ in each such $(N{-}1)$-space.
  We obtain a set $Z$ of size $|X| \cdot |Y|$.
  
  It remains to show that $Z$ is an $(r,s)$-set. 
  For this, consider a $s$-space $S$.
  Let $T_1, \ldots, T_m$ be the set of $(N{-}1)$-spaces
  in $Y$ which $S$ meets.
  As $Y$ is an $(r, s)$-set, $\sum_{i=1}^m (\dim(S \cap T_i) + 1) \leq r$.
  As $X$ is a proper $(r,r-1)$-set, $|S \cap T_i \cap Z| \leq \dim(S \cap T_i)+1$.
  As $S \cap Z = \bigcup_{i=1}^m S \cap T_i \cap Z$, $Z$ is a $(r,s)$-set.
\end{proof}

This way we obtain 
\begin{enumerate}[(i)]
 \item caps of size $(1+o(1)) q^{\frac{2}{3} n}$ in $\PG(n, q)$
for $n = 4^m-1$ by multiplying $m$ ovoids of $\PG(3, q)$
(This is best known, see \cite{EB1999});
 \item $(3,2)$-sets of size $(1+o(1)) q^{\frac{1}{3} n}$ in $\PG(n, q)$
 for $n = 4^m-1$ by multiplying $m$ rational normal curves of $\PG(4, q)$
 (Lemma \ref{lem:lower_bnd2}: $C q^{\frac{n}{3} - \frac23}$);
 \item $(3,2)$-sets of size $(1+o(1)) q^{\frac{n}{3} - \frac12}$ in $\PG(n, q)$
 for $n = 7 \cdot 4^m-1$ 
 by multiplying the Construction in Theorem \ref{thm:q32constr}
 with $m{-}1$ rational normal curves of $\PG(4, q)$;
 \item $(r,r-1)$-sets of size $(1+o(1)) q^{\frac{n+1}{c}}$ in $\PG(n, q)$
 for $n = (c+1)^m-1$ by multiplying $m$ rational normal curves of $\PG(c, q)$ 
 (Lemma \ref{lem:lower_bnd2}: $C q^{\frac{n+1}{c} - \frac{c-1}{c}}$).
\end{enumerate}

\section{Concluding Remarks}

Let us conclude with some open problems and conjectures.
For $q$ fixed and $n \rightarrow \infty$,
Ellenberg and Gijswijt famously showed that 
a cap has at most size $O(2.756^n)$ \cite{EG2017},
while a cap due to Edel has size $\Omega(2.2174^n)$ \cite{Edel2004}.
Note that the results by Ellenberg and Gijswijt have been 
generalized to $(r, r-1)$-sets in \cite{Bennett2019}.
For our setting of $q \rightarrow \infty$
and $n$ fixed, we believe the following to be true.

\begin{Conjecture}
  There exists a constant $\eps > 0$ such that
  the size of a cap in $\PG(n, q)$ is bounded by $O(q^{n-1-\eps})$.
\end{Conjecture}

\begin{Problem}
  Find a set of $\Omega(q^3)$ points
  in $\PG(4, q)$ with at most $O(q^{5})$ triples of collinear points.
\end{Problem}

Using \cite[Theorem 3]{DLR1994},
such an $(r, 1)$-set implies the existence 
of a cap of size $C' q^2 \sqrt{\log q}$ in $\PG(4, q)$.
This approach was suggested by Dhruv Mubayi.

The rational normal curve gives an 
example of size $q+1$ for a set of points with no 4 coplanar
(a $(3, 2)$-set), while in $\PG(4, q)$ the best known upper bound is $(\sqrt{2}+o(1)) q^{\frac32}$. In light of this, the following is the first 
important open problem in the investigation of $(r, s)$-sets.

\begin{Problem}
  Find a $(3,2)$-set in $\PG(4, q)$ of size $\Omega(q^{1+\eps})$.
\end{Problem}

\smallskip 

\paragraph*{Acknowledgment}
We thank David Conlon, Dhruv Mubayi, 
Leo Storme, Benny Sudakov, and Istv\'{a}n Tomon 
for comments on an earlier version of this document.
The first author is supported by a 
postdoctoral fellowship of the Research Foundation -- Flanders (FWO).

\end{document}